\documentclass[12pt, a4paper, english]{article}		
\usepackage[T1]{fontenc}
\usepackage{babel}		
\usepackage{fancyhdr}		
\usepackage[utf8]{inputenc}
\usepackage{amsmath}  
\usepackage{amsthm}		
\usepackage{amsfonts}		
\usepackage{amssymb} 
\usepackage{enumerate} 
\usepackage{graphicx} 
\usepackage{indentfirst}
\usepackage{url}
\newtheorem{teorema}{Theorem}
\newtheorem{lema}[teorema]{Lemma}
\newtheorem{proposicio}[teorema]{Proposition}
\newtheorem{corolari}[teorema]{Corollary}

\theoremstyle{definition}
\newtheorem{definicio}[teorema]{Definition}
\newtheorem{nota}[teorema]{Remark}
\newtheorem{exemple}[teorema]{Example}

\DeclareMathOperator{\C}{C}
\DeclareMathOperator{\E}{E}
\DeclareMathOperator{\F}{\mathbb{F}}

\DeclareMathOperator{\Aut}{Aut}
\DeclareMathOperator{\Ann}{Ann}

\DeclareMathOperator{\Ker}{Ker}

\DeclareMathOperator{\Soc}{Soc}

\DeclareMathOperator{\Z}{Z}

\title{On left braces in which every subbrace is an ideal}

\author{A. Ballester-Bolinches\thanks{Departament de Matem{\`a}tiques, Universitat de Val{\`e}ncia; Dr.\ Moliner, 50; 46100 Burjassot, Val\`encia, Spain; \texttt{Adolfo.Ballester@uv.es}, ORCID 0000-0002-2051-9075; \texttt{Ramon.Esteban@uv.es}, ORCID 0000-0002-2321-8139; \texttt{Vicent.Perez-Calabuig@uv.es}, ORCID 0000-0003-4101-8656} \and R. Esteban-Romero\addtocounter{footnote}{-1}\footnotemark \and L. A. Kurdachenko\thanks{Department of Algebra and Geometry, Oles Honchar Dnipro National University, Dnipro 49010, Ukraine; \texttt{lkurdachenko@gmail.com}; ORCID 0000-0002-6368-7319}\ \thanks{Part of the research of this author has been carried out in the Departament de Matem{\`a}tiques, Universitat de Val{\`e}ncia; Dr.\ Moliner, 50; 46100 Burjassot, Val\`encia, Spain} \and V. P{\'e}rez-Calabuig\addtocounter{footnote}{-3}\footnotemark}

\date{}
\begin{document}
\maketitle

\begin{abstract}
The aim of this paper is to introduce and study the class of all left braces in which every subbrace is an ideal. We call them {\it Dedekind} left braces. It is proved that every finite Dedekind left brace is centrally nilpotent. Structural results about Dedekind left braces and a complete description of those ones whose additive group is an elementary abelian $p$-group are also shown.  As a consequence, every 
multipermutational Dedekind left brace whose additive group is an elementary abelian $p$-group is multipermutational of level $2$.  A new class of left braces, the extraspecial left braces, is introduced and play a prominent role in our approach. 

\end{abstract}

  

\section{Introduction}
\label{sec:intro}

In his inspiring paper \cite{Drinfeld92}, Drinfeld posed the question of study the so-called set-theoretical solutions of the Yang-Baxter equation (YBE) which has given rise to an intensive research by means of algebraic and combinatorial methods (see \cite{EtingofSchedlerSoloviev99, Gateva-Ivanova04, LuYanZhu00}, for example). The main problem in this context is to find and classify all set-theoretical solutions that satisfy some prescribed properties which are determined by the action of  algebraic structures. Rump's \emph{left braces} stand out among them as a first exponent of a new algebraic structure introduced in \cite{Rump07} to deal with the classification of all involutive non-degenerate set-theoretical solutions of the YBE (solutions, for short).

A left brace is a nonempty set $A$ endowed with two group structures $(A,+)$ and $(A,\cdot)$, such that $(A,+)$ is abelian and
\[ a\cdot (b+c) = a\cdot b + a \cdot c - a, \ \text{for every $a,b,c\in A$.}\]
In particular, a left brace is called \emph{abelian} if both operations coincide. Every left brace provides a solution of the YBE, and viceversa, every solution of the YBE is a restriction of the solution associated to a left brace. Then, it is abundantly clear that the more we know about the algebraic structure of left braces the more we go forward in the classification problem of solutions. 

The illustrative survey of Cedó~\cite{Cedo18} features a selection of the most important algebraic developments on left braces with a great impact on the study of classes of solutions. In particular, the class of \emph{multipermutational solutions} sticks out as it admits an algebraic treatment by means of left-ordered groups and nilpotent notions of left braces (see~\cite{CedoJespersOkninski10, CedoJespersOkninski14, Gateva-IvanovaCameron12, Gateva-Ivanova18-advmath, Smoktunowicz18-tams} for further information). Moreover, this class has been object of an intensive study since it covers almost all solutions: ``among all involutive solutions with at most ten elements, more than 95\% are multipermutation'' (in~\cite{JespersVanAntwerpenVendramin23}). Consequently, the classification problem of solutions undoubtedly calls for an analysis of subfamilies of multipermutation solutions. In this light, it is particularly studied the case of multipermuation solutions of level~$2$ (v. gr. \cite{Gateva-IvanovaMajid11, JedlickaPilitowskaZamojska-Dzienio21, JedlickaPilitowska23, Rump22}).

Like every algebraic structure, the algebraic study of left braces is based on interactions between substructures (subbraces) and quotients (provided by the so-called ideals). In this paper, we introduce the brace theoretic analogue of Dedekind groups, namely \emph{Dedekind left braces} or left braces for which every subbrace is an ideal (see Section~\ref{sec:dedekind}). One of our main motivations comes from the well-known fact that Dedekind groups form a family of groups with nilpotency class~$2$ (it follows from the structural theorems proved by Dedekind~\cite{Dedekind97} and Baer~\cite{Baer33}). A first glance to examples of Dedekind left braces in Section~\ref{sec:dedekind} shows that a structural theorem for them is a challenging open problem. In this direction, our first main result shows that every finite Dedekind left brace is centrally nilpotent (see Theorem~\ref{teo:Dedekind->central-nilp} and Section~\ref{sec:prelim} for preliminary notions on nilpotency of left braces). As a consequence, every solution associated with a finite left brace is multipermutational.

The fact that the subbrace structure coincides with the ideal structure of a left brace not only has an algebraic interest by their own, but also could shed light on the study of simple and soluble left braces and associated simple and soluble solutions (see~\cite{BallesterEstebanJimenezPerezC24-solubleskewbraces} and~\cite{CedoOkninski21} for further information). Delving into the study of Dedekind left braces, one catches a glimpse of the importance of getting a tighter grip of the abelian nature of the left brace structure. According to the classification of finite abelian groups, elementary abelian $p$-groups can be considered as atomic abelian structures. Therefore, it is natural to think about them as a first remarkable family for which a structural theorem can be proved. Section~\ref{sec:dedekind_elem-abel} is devoted to prove such structural theorem: see  Theorems~\ref{teo:descomp-strong->Dedekind} and~\ref{teo:Ahypersoc+Dedekind->descomp-strong}, and Corollary~\ref{cor:Ahypersoc+finit++Dedekind->E0-1-2+polinomis}. These results hinge on two main facts. On one hand, we introduce extraspecial left braces which play a key role in the proof of our results. On the other hand, we prove that 
multipermutational Dedekind left braces whose additive group is an elementary abelian $p$-group have multipermutational level~$2$. It follows that finite Dedekind left braces whose additive group is an elementary abelian $p$-group have central nilpotency length~$2$. Consequently, remarkable families of left braces whose associated multipermutation solutions have level~$2$ are completely described.

\section{Preliminaries}
\label{sec:prelim}

In this section, we fix notation and give some background concepts and results on braces. Although many of them can be considered folklore, others are new. We also include a subsection on results of bilinear forms which are crucial to prove our main results.

In this paper, products on left braces will be denoted by juxtaposition and come before sums on left braces. Left braces $(A,+,\cdot)$ will be referred simply as $A$, whenever the operations $+$ and $\cdot$ are clear from the context. We shall denote with $(A,+)$ (resp. $(A,\cdot)$) the additive (resp. multiplicative) group of a left brace $A$, and we denote by $0$ the common identity element in $(A,+)$ and $(A, \cdot)$. If $A$ is finite, we denote $\pi(A)$ the set of primes dividing the order $|A|$ of $A$.

Following Rump~\cite{Rump07}, left braces can be seen as a generalisation of Jacobson radical rings by means of the so-called \emph{star operation }
\[ \ast \colon (a,b)\in A \times A \longmapsto a \ast b = ab - a - b \in A.\]
We write star products preceding both products and sums in left braces. Moreover, there exists an action of the multiplicative group on the additive group, usually denoted as the homomorphism $\lambda\colon a \in (A,\cdot)\mapsto \lambda_a \in \Aut(A,+)$, such that $ab = a + \lambda_a(b)$ and $a \ast b = \lambda_a(b) - b$ links all operations in $A$. In particular, the following equalities are easily verified and are widely used throughout the paper:
\begin{eqnarray}
\label{ast_distesq} a \ast (b+c) & = & a \ast b + a \ast c\\
\label{ast_prod} (ab)\ast c & = & a \ast (b \ast c) + b\ast c + a \ast c
\end{eqnarray} 

Substructures in left braces are classified according to each embeddings in their additive and multiplicative groups. Let $A$ be a left brace. A non-empty subset $S$ of $A$ is called a \emph{subbrace} if $(S,+)$ and $(S,\cdot)$ are subgroups of $(A,+)$ and $(A,\cdot)$, respectively. We write $S \leq A$. A subbrace $I$ of $A$ is called an \emph{ideal} if it is $\lambda$-invariant and $(I, \cdot) \unlhd (A,\cdot)$. Ideals are characterised as those subbraces $I$ of $A$ such that $A \ast I, I \ast A \subseteq I$. 
If $I$ is an ideal of $A$ then $a+I = aI$, and therefore, a quotient left brace structure $(A/I, +, \cdot)$ is provided. Moreover, it holds $(a + I)\ast(b+I) = (a\ast b) + I$, for every $a,b\in A$.

\begin{nota}
\label{nota:semidirect-brace}
Let $A$ be a left brace. Following~\cite{BallesterEsteban22}, we can construct $G_A = [(A,+)]_\lambda(A,\cdot)$ the semidirect product associated with the action $\lambda$: it holds $(a,b)(c,d) = (a + \lambda_b(c), bd)$ for every $a,b,c,d\in A$. Then, we have that for every $a,b\in A$, the commutator
\[ [(0,a),(b,0)] = (0,a)(b,0)(0,a^{-1})(-b,0) = (\lambda_a(b) - b, 0) = (a\ast b, 0)\]
It follows that $I$ is an ideal of $A$ if, and only if, $I \times I = \{(i,j)\mid i,j\in I\} \unlhd G_A$.
\end{nota}

Let $S$ be a subset of a left brace $A$. We denote by $\langle S \rangle$ the \emph{subbrace generated} by $S$, which is the intersection of all subbraces of $A$ containing $S$, and $\langle S\rangle_+$, $\langle S \rangle_{\boldsymbol{\cdot}}$ denote the subgroups generated by $S$ in the additive and the multiplicative group of $A$, respectively. If $S$ is a singleton $\{a\}$, we just write $\langle a\rangle$ for the corresponding generated substructure. If $X$ and $Y$ are subsets of $A$, $X \ast Y = \langle x \ast y\mid x \in X, \, y \in Y\rangle_+$. A left brace $A$ is said to be \emph{abelian} if $A \ast A = 0$, or equivalently, $ab = a+b= b+a = ba$ for every $a,b\in A$. We say that a subbrace (resp. ideal, quotient left brace) is abelian if it is abelian as a left brace structure. For abelian left braces, both group structures coincide, and therefore, we can refer to them as \emph{isomorphic} to the common group structure.

\begin{lema}
\label{lema:a·a=0ciclic}
Let $A$ be a left brace and let  $a \in A$ such that $a\ast a = 0$. Then, $\langle a \rangle = \langle a \rangle_+ = \langle a \rangle_{\boldsymbol{\cdot}}$ is an abelian subbrace of~$A$.  
\end{lema}

\begin{proof}
Since $a\ast a = 0$, it holds $a^2 = 2a$. Then, we see that
\[ a^3 = aa^2 = a(a+a) = a^2 + a^2 - a = 2a + 2a - a = 3a\]
By induction, it follows $a^n = na$ for all positive integers $n$. Moreover,
\[ a = a^{-1}a^2 = a^{-1}(2a) = a^{-1}(a+a) = a^{-1}a + a^{-1}a - a^{-1} = 0 + 0 -a^{-1} = -a^{-1}\]
Thus, $a^{-1} = -a$. Then,
\[ a^{-2} = a^{-1}a^{-1} = a^{-1}(-a) = a^{-1}(0-a) = a^{-1} - a^{-1}a + a^{-1} = -a - 0 - a = -2a\]
By induction, we obtain that $a^{-n} = -na$, for all positive integer~$n$. Therefore, it follows that $(ka)(na) = (k+n)a$ and $(ka)^{-1} = -ka$, for all integers $k,n$. This shows that $\langle a \rangle = \langle a \rangle_+ = \langle a \rangle_{\boldsymbol{\cdot}}$ is an abelian subbrace of~$A$. 
\end{proof}

\emph{Homomorphisms} between two left braces $A$ and $B$ are mappings conserving sums and products; $A$ and $B$ are called \emph{isomorphic} if there exists an \emph{isomorphism} (i.e. a bijective homomorphism) between them; and the \emph{direct product} $A \oplus B$ is constructed as the direct product of each sum and product of $A$ and~$B$.

%

\bigskip

The following particular substructures play a key role in the structural study of a left brace~$A$. Given $\emptyset \neq S \subseteq A$, the following sets
\begin{align*}
 \Ann_A^l(S) & = \{a \in A \mid a \ast x = 0,\, \text{for all $x\in S$}\} = \\
& =\{a \in A \mid ax = a+x,\,\text{for all $x\in S$}\};\\
\Ann_A(S) & = \{a \in A\mid a \ast x = x \ast a = 0, \, \text{for all $x\in S$}\} = \\
& =\{a \in A \mid ax = a+x = xa,\,\text{for all $x\in S$} \} = \Ann_A^l(S) \cap \C_{(A,\cdot)}(S);
\end{align*}
are respectively defined as the \emph{left annihilator} 
and the \emph{annihilator} of $S$. In case that $S = A$, they are respectively called the \emph{socle} $\Soc(A)$
and the \emph{centre} $\zeta(A)$ of the left brace $A$ (see \cite{BonattoJedlicka23}, \cite{CatinoColazzoStefanelli19} or \cite{CedoSmoktunowiczVendramin19}, for example). It follows that a left brace $A$ is abelian if, and only if, $\zeta(A) = A$. Moreover, every subbrace of the centre of a left brace is an abelian ideal and every subbrace of the socle is abelian. 

\begin{nota}
\label{nota:centre-semidirect}
Let $A$ be a left brace. According to Remark~\ref{nota:semidirect-brace}, it follows that $a \in \zeta(A)$ if, and only if, $(a,0), (0,a) \in \zeta(G_A)$.
\end{nota}

The following lemma shows important properties of annihilators in a brace. Although it can be derived from previous works (e.g. \cite[Section 3]{ColazzoFerraraTrombetti-publ-math}), we include a proof for the sake of completeness.

\begin{lema}
\label{lema:left-right-ann_substruct}
Let $A$ be a left brace and let $\emptyset \neq S\subseteq A$. Then, the following hold
\begin{enumerate}
\item $\Ann_A^l(S) \leq (A, \cdot)$. Moreover, if $S$ is an ideal, $(\Ann_A^l(S),\cdot)\unlhd (A,\cdot)$.
\item $\Ann_A(S)$ is a subgroup of the centraliser of $S$ in the multiplicative group $(A,\cdot)$. Moreover, if $S$ is an ideal, then $(\Ann_A(S), \cdot) \unlhd (A,\cdot)$. 
\end{enumerate}
\end{lema}

\begin{proof}
1. Let $a,b \in \Ann_A^l(S)$. For every $x\in S$, it holds
\[  x = a^{-1}ax = a^{-1}(a+x) = 0 - a^{-1} + a^{-1}x,\]
i.e. $a^{-1}x = a^{-1} + x$. Moreover, $(ab)x = a(bx) = a(b+x) = ab + ax - a = ab + x$. Thus, $(\Ann_A^l(S), \cdot) \leq (A,\cdot)$. 

Assume that $S$ is an ideal. Then, we can consider the restriction of the $\lambda$-action $\bar{\lambda}\colon a \in A \rightarrow \left. \lambda_a\right|_S \in \Aut(S,+)$. Thus, $a \in \Ker \bar{\lambda}$ if, and only if, $\lambda_a(s) = s$ for every $s\in S$, i.e. $a \in \Ann_A^l(S)$. Hence, $\Ann_A^l(S) = \Ker \bar{\lambda} \unlhd (A,\cdot)$.


2. By the previous statement, $\Ann_A(S)$ is a subgroup of the centraliser of $S$ in the multiplicative group $(A,\cdot)$. Suppose that $S$ is an ideal of $A$. Since $(S,\cdot)\unlhd (A,\cdot)$ it follows that the centraliser of $S$ in the multiplicative group $(A,\cdot)$ is a normal subgroup. 

Let $b \in A$ and $a \in \Ann_A(S)$. Thus, $bab^{-1}\in \Ann_A^l(S)$ and $(bab^{-1})x = x(bab^{-1})$, for every $x\in S$. Hence,
\[ x(bab^{-1}) = (bab^{-1})x = bab^{-1} + x = x + bab^{-1}, \]
i.e. $bab^{-1}\in \Ann_A(S)$.
\end{proof}

In particular, it follows that the socle and the centre of a left brace are ideals. 


Let $A$ be a left brace. The left and right series of $A$ are respectively defined as 
\begin{align*}
& A^1 = A  \geq A^2 = A \ast A \geq \ldots \geq A^{n+1} = A \ast A^n \geq \ldots\\
& A^{(1)} = A \geq A^{(2)} = A \ast A \geq \ldots \geq A^{(n+1)} = A^{(n)}\ast A \geq \ldots 
\end{align*}
Following~\cite{Rump07}, $A$ is said to be \emph{left (right) nilpotent} if the left (right) series reaches $0$. Then, $A$ is said to be \emph{centrally nilpotent} or \emph{nilpotent in the sense of Smoktunowicz} if it is left and right nilpotent (see  \cite{BallesterEstebanFerraraPerezCTrombetti-arXiv-cent-nilp,BonattoJedlicka23, CedoSmoktunowiczVendramin19} for further information). 

Now, we consider
\begin{align*}
& 0 = \Soc_0(A) \leq \Soc_1(A) \leq \ldots \leq \Soc_{n}(A) \leq \Soc_{n+1}(A) \leq \ldots \\
& 0 = \zeta_0(A)  \leq \zeta_1(A)  \leq \ldots \leq \zeta_n(A) \leq \zeta_{n+1}(A)\leq \ldots 
\end{align*}
ascending series of ideals defined by the recursive rule:
\[ \begin{array}{ll}
 \Soc_{n+1}(A)/\Soc_n(A) = \Soc(A/\Soc_n(A)),& \zeta_{n+1}(A)/\zeta_n(A) = \zeta(A/\zeta_n(A)); \end{array}\]
for every non-negative integer $n$. These are respectively called the \emph{socle (resp. upper central) series} of $A$. A left brace is said to have \emph{finite multipermutational level} if $\Soc_n(A) = A$ for some $n$. 


\subsection*{Bilinear forms}
\label{subsec:bilinear}

We end the section of preliminaries with a subsection which features some concepts and results on bilinear forms over finite dimensional vector spaces. 
\emph{From now on, $V$ denotes a finite dimensional vector space over a field $\mathbb{F}$ and $\phi$ a bilinear form on $V$.}

Let $U \leq V$. We write
\[ 
\begin{array}{l}
^\perp U = \{x \in V  \mid \phi(x,u) = 0, \, \text{for all elements $u\in U$}\},\\
U^\perp = \{x \in V \mid \phi(u,x) = 0,\, \text{for all elements $u \in U$}\}.\end{array}\]
It turns out that $^\perp U$ and $U^\perp$ are subspaces of $V$, which are called the \emph{left (resp. right) orthogonal complement} of $U$ in $V$. 

The following result is well-known.

\begin{proposicio}
\label{prop:phi_non-degenerate_direct_sum}
Given $U \leq V$,  it follows that
\[ \dim_{\F} U + \dim_{\F}\, ^\perp U - \dim_{\F}(U \cap V^\perp) = \dim_{\F} V\]
Moreover, if the restriction of $\phi$ over $U$ is non-degenerate, then~$V =U\oplus\,^\perp U$.
\end{proposicio}

A bilinear form $\phi$ is called \emph{strong non-degenerate} if the restriction of $\phi$ over every non-zero subspace is non-degenerate.

\begin{corolari}
\label{cor:2.11strongnondegenerate_matrice triangular}
Assume that  $\phi$ is strong non-degenerate on $V$. Then, $V$ has a basis in which the matrix of $\phi$ is upper triangular.
\end{corolari}

\begin{proof}
Let $n=\dim_{\F} V$ and assume the assertion true for every vector space of dimension less than $n$ with $n > 1$. Take $U_1 = \langle u_1\rangle$ for some $0\neq u_1\in V$. Since the restriction of $\phi$ over $U_1$ is non-degenerate, Proposition~\ref{prop:phi_non-degenerate_direct_sum} shows that $V = U_1\oplus\, ^\perp U_1$. By induction, there exists a basis $\{u_2, \ldots, u_n\}$ of $^\perp U_1$ in which the basis of $\phi|_{^\perp U_1}$ is upper triangular. Hence, $\{u_1, u_2, \ldots, u_n\}$ is a basis of $V$ in which the matrix of $\phi$ is upper triangular.
\end{proof}

\begin{corolari}
\label{cor:2.12phi(x,x)neq0_triangularmatrix}
Assume that $\phi(x,x) \neq 0$ for every $0\neq x \in V$. Then, $V$ has a basis in which the matrix of $\phi$ is upper triangular.
\end{corolari}

\begin{corolari}
\label{cor:no-isotrop->dim<=2}
Suppose that $V$ is a vector space over a finite field $\F_q$, with $q = p^t$ for some prime $p$. Let $\phi$ be a bilinear form on $V$ such that $\phi(x,x) \neq 0$ for every $0\neq x \in V$. Then, $V$ has dimension at most $2$.
\end{corolari}

\begin{proof}
Assume that $\dim_{\F_q}(V) > 2$. By Corollary~\ref{cor:2.12phi(x,x)neq0_triangularmatrix}, we can take a subspace $S \leq V$ with a basis $\{b_1, b_2, b_3\}$ such that $\phi(b_2,b_1) = \phi(b_3,b_1) = 0$. Set $\sigma_{i,j} := \phi(b_i,b_j)$, for $1\leq i \leq j \leq 3$.  Let $x = \lambda_1b_1+\lambda_2b_2 + \lambda_3b_3 \in S$ with $\lambda_i \in \F_q$, $1\leq i \leq 3$. Then, it holds
\[ \phi(x,x) = \sum_{1 \leq i \leq j \leq 3} \lambda_i\lambda_j \sigma_{i,j} \]
By Chevalley's Theorem (see \cite[Corollary 6.6]{LidlNiederreiter97}, for example), 
\[ \sum_{1 \leq i \leq j \leq 3}\sigma_{i,j}X_iX_j\]
has a non-zero solution. Therefore, there exists $0 \neq x\in V$ such that $\phi(x,x) = 0$, a contradiction.
\end{proof}

\section{Dedekind left braces}
\label{sec:dedekind}

Our definition of Dedekind left braces is inspired by the correspondent notion of Dedekind groups. 

\begin{definicio}
A left brace $A$ is said to be \emph{Dedekind} if every subbrace of $A$ is an ideal.
\end{definicio}

It is clear that the property of being Dedekind is inherited by subbraces and quotients.

\begin{nota}
\label{nota:no-proper-subbrace}
Obviously, every brace with no proper subbraces is Dedekind. Theorem~A in \cite{BallesterEstebanJimenezPerezC24-solubleskewbraces} shows that finite left braces with no proper subbraces are abelian left braces isomorphic to a cyclic group of prime order. As being a subbrace is a transitive property in a left brace, by induction, this implies that every finite left brace has a subbrace of prime order.
\end{nota}

Of course, abelian left braces are Dedekind, but there are left braces whose multiplicative group is isomorphic to the additive group which are not Dedekind: the left brace corresponding to \texttt{SmallBrace(8, 7)} in the \textsf{Yang--Baxter} library~\cite{VendraminKonovalov22-YangBaxter-0.10.2} for \textsf{GAP}~\cite{GAP4-12-2} has both additive and multiplicative groups isomorphic to the direct product of cyclic groups $\C_4\times \C_2$ but it is not Dedekind as it has a subbrace of order $2$ which is not ideal.

This example also shows that having both the additive and multiplicative groups Dedekind it is not sufficient for a left brace to be Dedekind. This situation changes if we assume the additive group to be cyclic.

\begin{proposicio}
\label{prop:ciclic+dedekind->dedekind}
Let $A$ be a left brace with a cyclic additive group and a Dedekind multiplicative group. Then, $A$ is Dedekind.
\end{proposicio}

\begin{proof}
Since $(A,+)$ is cyclic, it follows that  every subgroup of $(A,+)$ is $\lambda$-invariant. Thus, every subgroup of $(A,+)$ is a subbrace of $A$. Since both $(A,+)$ and $(A,\cdot)$ are Dedekind groups, we conclude that every subbrace is an ideal.
\end{proof}

The condition on the multiplicative group can not be removed in general: there exists a left brace with additive group isomorphic to  the cyclic group $\C_6$ and multiplicative group isomorphic to the symmetric group $\text{Sym}(3)$ such that the cyclic subgroup of order $2$ in $\C_6$ is a subbrace but not an ideal. 

Moreover, Proposition~\ref{prop:ciclic+dedekind->dedekind} can not been extended for one-generated braces. In~\cite{BallesterEstebanKurdachenkoPerezC-arXiv-actions-ab-group-braces}, one-generated left braces $A$ with $\Soc_2(A) = A$ are shown to admit useful descriptions. However, this condition it is not sufficient to yield Dedekind left braces.

\begin{exemple}
Let $A$ be the left brace corresponding to \texttt{SmallBrace(16, 72)} in the \textsf{Yang--Baxter} library~\cite{VendraminKonovalov22-YangBaxter-0.10.2} for \textsf{GAP}~\cite{GAP4-12-2}. It has additive group $(A,+)$ isomorphic to $\C_4\times \C_4$, and multiplicative group $(A,\cdot)$ isomorphic to $\C_8 \times \C_2$ which is a Dedekind group. Moreover, $\Soc_2(A) = A$ but $A$ is not a Dedekind left brace since it has a subbrace of order $2$ which is not an ideal. Indeed, it is a minimal counterexample. In addition, the left brace $A$ corresponding to \texttt{SmallBrace(81,167)} serves also as a minimal counterexample of odd order.
\end{exemple}

The previous examples evince that there is a long way off to find a characterisation for Dedekind left braces. In this section, we present a first remarkable result for the case of finite left braces.

\begin{teorema}
\label{teo:Dedekind->central-nilp}
Every finite Dedekind left brace is centrally nilpotent. 
\end{teorema}

We work towards a proof of Theorem~\ref{teo:Dedekind->central-nilp}.

%
%

\begin{proposicio}
\label{prop:ded->descomp-sylow}
Let $A$ be a finite Dedekind left brace. For every $p \in \pi(A)$, the Sylow $p$-subgroup $I_p$ of $(A,+)$ is an ideal of $A$. In particular, $A = \bigoplus_{p\in \pi(A)} I_p$.
\end{proposicio}

\begin{proof}
Let $p \in \pi(A)$ and let $(I_p,+)$ be the Sylow $p$-subgroup of $(A,+)$. Since it is a characteristic subgroup, it is $\lambda$-invariant, and therefore, a subbrace of~$A$. Hence, $I_p$ is an ideal as $A$ is Dedekind.
\end{proof}

\begin{proposicio}
\label{prop:Dedekind->orderp<Soc}
Let $A$ be a Dedekind left brace such that both $(A,+)$ and $(A,\cdot)$ are $p$-groups, $p$ a prime. Then every subbrace $S$ of order $p$ is contained in $\zeta(A)$.
\end{proposicio}

\begin{proof}
Let $S$ be a subbrace of $A$ of order $p$, $p$ a prime. Then, $S$ is an ideal of $A$. Thus, $(S,\cdot)$ is a minimal normal subgroup of $(A,\cdot)$ and, therefore, $(S,\cdot) \leq \Z(A,\cdot)$.

Following Remark~\ref{nota:semidirect-brace}, take $G_A = [(A,+)]_\lambda(A,\cdot)$ the semidirect product associated with $(A,+)$. Since $S$ is an ideal of $A$, it follows that $\{(s,0)\mid s\in S\}$ is a minimal normal subgroup of $G_A$. Thus, $\{(s,0)\mid s \in S\} \leq \Z(G_A)$. Then, for every $a\in A$,
$(0,0) = [(0,a),(s,0)]= (a \ast s, 0)$, i.e. $a \ast s = 0$ and, therefore, 
\[ \lambda_s(a) = -s + sa = as - s = a + a \ast s = a.\]
Moreover, for every $a\in A$ we also have that $[(a,0),(0,s)] = (0,0)$. Thus, $\{(0,s)\mid s\in S\} \leq \Z(G_A)$. Hence, $S \leq \zeta(A)$ by Remark~\ref{nota:centre-semidirect}.
\end{proof}

\begin{corolari}
\label{cor:dedekind+pbrida->rightnilpotent}
Let $A$ be a finite Dedekind left brace of order $p^k$ for some prime $p$ and positive integer $k$. Then, $A$ is centrally nilpotent.
\end{corolari}

\begin{proof}
Let $A$ be a minimal counterexample. If $A$ has not any subbrace $S$ of order $p$, then by Remark~\ref{nota:no-proper-subbrace}, $A$ is trivial isomorphic to cyclic group of order $p$, which is not possible. Then, there exists $S \lneq A$ of order $p$. According to Proposition~\ref{prop:Dedekind->orderp<Soc}, it holds $S \leq \zeta(A)$. Hence, by minimality, $A/\zeta(A)$ is centrally nilpotent and, therefore, $A$ is centrally nilpotent; a contradiction.
\end{proof}

\medskip

\begin{proof}[Proof of Theorem~\ref{teo:Dedekind->central-nilp}]
By Theorem~4.13 in~\cite{BallesterEstebanFerraraPerezCTrombetti-arXiv-cent-nilp}, it follows that a left brace $A$ is centrally nilpotent if, and only if, the Sylow $p$-subgroups $I_p$ of $A$ are centrally nilpotent ideals and $A = \bigoplus_{p\in \pi(A)} I_p$. The result follows from Proposition~\ref{prop:ded->descomp-sylow} and Corollary~\ref{cor:dedekind+pbrida->rightnilpotent}.
\end{proof}

In case of left braces with an elementary abelian additive $p$-group, $p$ a prime, we can say more.

\begin{proposicio}
\label{prop:dedekind->Soc_a=Z_a}
Let $A$ be a Dedekind left brace. If $(A,+)$ is elementary abelian $p$-group, then $\Soc_{n}(A) = \zeta_n(A)$ for every positive integer $n$.
\end{proposicio}

\begin{proof}
By induction, it suffices to show that $\Soc(A) = \zeta(A)$. But this follows as a consequence of Proposition~\ref{prop:Dedekind->orderp<Soc}, as every cyclic subgroup of $\Soc(A)$ is a subbrace, and therefore, it is contained in $\zeta(A)$.
\end{proof}




\subsection*{Extraspecial left braces}
\label{subsec:extraspecial}

We introduce here a family of Dedekind left braces. These will play a key role in Section~\ref{sec:dedekind_elem-abel}. The following definition of extraspecial left braces is motivated by the corresponding one in groups.  

\begin{definicio}
A non-abelian left brace $E$ is called \emph{extraspecial} if its additive group is an elementary abelian $p$-group, $p$ a prime, and there exists $c\in \zeta(E)$ such that $0 \neq E/\langle c \rangle_+$ is an abelian left brace. An extraspecial brace $E$ is called \emph{strong extraspecial} if $a \ast a \neq 0$ for every element $a\notin \langle c \rangle_+$.
\end{definicio}

Extraspecial left braces form a subfamily of important families of braces with a great impact on the structural theory of braces and solutions of the Yang-Baxter equation. That is the case of almost polycyclic braces and supersoluble braces (see, respectively, \cite{BallesterEstebanFerraraPerezCTrombetti-polycyclicbraces} and~\cite{BallesterEstebanFerraraPerezCTrombetti24}).

\begin{nota}
\label{nota:extrasp-e.v}
Let $E$ be an extraspecial brace such that $(E,+)$ is an elementary abelian $p$-group, for some $p$ prime. For every element $c \in \zeta(E)$, the additive subgroup $C:= \langle c\rangle_+$ is an abelian ideal which is isomorphic to a cyclic group of order~$p$. Then, $C$ can be embedded in the field of $p$-elements $\F_p$ as $kc \mapsto k$, for every $0 \leq k < p$. Moreover, $E/C$ can also be regarded as a vector space over the field $\F_p$. \emph{From now on $\F_p$ denotes the field of $p$-elements with $p$ a prime}. 
\end{nota}

\begin{nota}
\label{nota:extrasp-socle=centre}
If $E$ is an extraspecial left brace, $E \neq \Soc(E)$, and then $\Soc_2(E) = E$, since $a \ast b \in \langle c \rangle_+ \leq \Soc(E)$ for every $a,b\in E$. Thus, $E\ast E = \langle c \rangle_+$. Moreover, if $E$ is strong extraspecial, $\zeta(E) = \Soc(E)$, as $a \notin \zeta(E)$ implies that $a \ast a \neq 0$. Therefore, $\Ann_E(\Soc(E)) = E$.
\end{nota}

It follows that (strong) extraspecial left braces are characterised by the existence of (strong non-degenerate) bilinear forms.

\begin{teorema}
\label{teo:(strong)extras-(strong)bilinear}
Let $E$ be a (strong) extraspecial left brace and let $C:= \langle c\rangle_+ \leq \zeta(E)$ such that $E/C$ can be regarded as a vector space over $\F_p$ for some prime~$p$. The map $\phi \colon  E/C \times E/C  \rightarrow \F_p$ defined as:
\[ (xC, yC)  \mapsto  k_{x,y}, \ \text{where $x\ast y = k_{x,y}c$,}\ \text{for every $xC,yC \in E/C$},  \]
is a (strong non-degenerate) bilinear form. 

Conversely, let $V$ be a vector space over $\F_p$, let $C = \langle c \rangle_+$ be an additive cyclic group of order $p$ and let $\phi$ be a (strong non-degenerate) bilinear form on $V$. The group $A = (V,+) \oplus C$ endowed with the product
\[ (x,k)(y,t) := (x + y, k+t+\phi(x,y)c), \quad \text{for every $x,y\in V$, $k,t\in C$}\]
turns out to be a (strong) extraspecial left brace.
\end{teorema}

\begin{proof}
Recall that Remark~\ref{nota:extrasp-socle=centre} states that $x \ast y = k_{x,y}c \in C$, for every $x,y\in E$. First, we check that $\phi$ is well-defined. Let $x_1C, x_2C, y_1C,y_2C \in E/C$ such that $x_1C = x_2C$ and $y_1C = y_2C = y_2 + C$. Then, $x_2 = x_1c$, $y_2 = y_1+c'$, for some elements $c, c'\in C$. Since $C \leq \zeta(E)$, it holds
\begin{align*}
 x_2\ast y_2 & = (x_1c) \ast (y_1+c') = (x_1c) \ast y_1 + (x_1c) \ast c' = (x_1c) \ast y_1 = \\
& = x_1\ast (c\ast y_1) +  c\ast y_1 + x_1\ast y_1  = x_1\ast y_1.
\end{align*}

Now, we see that $\phi$ is bilinear. Given $xC, yC, zC \in E/C$, we have that $xC+zC = (x+z)C = xzC$, as $E/C$ is abelian. Thus, $x+z = xzc$, for some element $c\in C$. Applying that $z\ast y \in \zeta(E)$, it holds
\begin{align*}
k_{x+z,y}c & = (x+z)\ast y = (xzc)\ast y = (xz) \ast (c\ast y) + c \ast y + (xz) \ast y  = \\
& = (xz) \ast y = x \ast (z\ast y)  + z \ast y + x\ast y  = x\ast y + z \ast y = k_{x,y}c +
 k_{z,y}c \\
 & =  (k_{x,y}+k_{z,y})c.
 \end{align*}
Similarly,
\[ k_{x,y+z}c = x \ast (y+z) = x \ast y + x \ast z = k_{x,y}c + k_{x,z} c = (k_{x,y}+k_{x,z})c.\]
Moreover, if $E$ is strong extraspecial, then $x \ast x \neq 0$ for every $x \notin C$, which means $\phi(x+C, x+C) \neq 0$. Thus, $\phi$ is strong non-degenerate.

Conversely, assume that $(A,+) = (V,+) \oplus C$, where $(V,+)$ is an elementary abelian $p$-group and $C = \langle c \rangle_+$ is a cyclic group of order $p$, and the product
\[ (x,k)(y,t) := (x + y, k+t+\phi(x,y)c), \quad \text{for every $x,y\in V$, $k,t\in C$}\]
is defined for some bilinear form $\phi\colon V \times V \rightarrow \F_p$.

Only the associative law of the product and the left brace distributivity law are in doubt. Let $(x,k), (y,t),(z,s) \in A$. Then
\begin{align*}
\big((x,k)(y,t)\big)\big(z,s\big) & = \big(x+y, k+t+\phi(x,y)c\big)\big(z,s\big) = \\
& = \big(x + y + z, k + t + \phi(x,y)c + s + \phi(x+y,z)c\big) = \\
& = \big(x+y+z, k+t+s + \phi(x,y)c + \phi(x,z)c + \phi(y,z)c\big);\\
\big(x,k\big)\big((y,t)(z,s)\big) & = \big(x, k\big)\big(y+z, t+s+\phi(y,z)c\big) = \\
& = \big(x+y+z, k+t+s+\phi(y,z)c + \phi(x,y+z)c \big) = \\
& = \big(x+y+z, k+t+s+\phi(y,z)c + \phi(x,y)c + \phi(x,z)c\big);
\end{align*}
i.e. the product is associative.

Furthermore,
\begin{align*}
 \big(x,k\big)\big((y,t)+(z,s)\big) & = (x,k)(y+z,t+s) = \\ 
& = \big(x+y+z, k+t+s + \phi(x,y+z)c\big) = \\
& = \big(x+y+z, k+t+s + \phi(x,y)c + \phi(x,z)c\big);
\end{align*}
and
\begin{align*}
& (x,k)(y,t) + (x,k)(z,s) - (x,k) = \\
& = \big(x+y, k+t+\phi(x,y)c\big) + \big(x+z,k+s+\phi(x,z)c\big) - (x,k) = \\
& = \big(x+y+z, k+t+s + \phi(x,y)c + \phi(x,z)c\big).
\end{align*}
Hence, $(A,+,\cdot)$ is a left brace.
\end{proof}

The following example provides three families of finite extraspecial left braces.

\begin{exemple}
\label{ex:extrasp}
Let $p$ be a prime and $\mathbb{F}_p$ the field of $p$ elements. We define the following sets with their corresponding addition and product:
\begin{enumerate}
\item Let $\E_0(m,p) = \mathbb{F}_p\oplus \mathbb{F}_p$, with $0 < m < p$. For every $k_1,k_2,t_1,t_2\in \mathbb{F}_p$, we define
\begin{align*}
 (k_1,k_2) + (t_1,t_2) & := (k_1+t_1, k_2+t_2),\\
 (k_1,k_2)(t_1,t_2) & := (k_1+t_1, k_2+t_2+ mk_1t_1)
\end{align*}
By Theorem~\ref{teo:(strong)extras-(strong)bilinear}, $\E_0(m,p)$ is a strong extraspecial left brace, since the map $\phi\colon \F_p \times \F_p \rightarrow \F_p$, given by $\phi(k,t) = mkt$ for every $k,t\in \F_p$, is a strong non-degenerate bilinear form. Moreover, given $(k_1,k_2), (t_1, t_2) \in \E_0(m,p)$, it holds
\[ (k_1,k_2) \ast (t_1, t_2) = (0, mk_1t_1)\]
Therefore, $\zeta(\E_0(m,p)) = \langle (0,1) \rangle_+$.

\item Let $\E_1(m,p) = \mathbb{F}_p\oplus \mathbb{F}_p\oplus \mathbb{F}_p$, with $0 \leq m < p$. For every $k_1,k_2,k_3, t_1,t_2,t_3\in \mathbb{F}_p$, we define
\begin{align*}
(k_1,k_2,k_3)+(t_1,t_2,t_3) & := (k_1+t_1,k_2+t_2,k_3+t_3) \\
(k_1,k_2,k_3)(t_1,t_2,t_3) & := (k_1+t_1, k_2+t_2,k_3+t_3+ mk_1t_1+k_2t_2)
\end{align*}
By Theorem~\ref{teo:(strong)extras-(strong)bilinear}, $\E_1(m,p)$ is an extraspecial left brace, since the map $\phi\colon \F_p^2 \times \F_p^2 \rightarrow \F_p$, given by 
\[ \phi((k_1,k_2),(t_1,t_2)) = mk_1t_1 + k_2t_2\] is a bilinear form. Moreover, it holds
\[ (k_1,k_2,k_3) \ast (t_1, t_2,t_3) = (0, 0, mk_1t_1 + k_2t_2)\]
If $m = 0$, $\zeta(\E_1(m,p)) = \langle (1,0,0),(0,0,1)\rangle_+$. Otherwise, $\zeta(\E_1(m,p)) = \langle (0,0,1) \rangle_+$. Moreover, $a \ast a \neq 0$ for every $a = (k_1,k_2,k_3) \notin \zeta(\E_1(m,p))$ if, and only if, $m \neq 0$ and $mk_1^2 + k_2^2 \neq 0$ for every $k_1,k_2 \in \mathbb{F}_p$ not both zero. Of course, if $k_2 \neq 0$ this is equivalent to say that the polynomial $mX^2 + 1$ has no roots in $\mathbb{F}_p$; while if $k_2 = 0$, then $mk_1^2 = 0$ and $m \neq 0$ imply that $k_1 = k_2 = 0$, which is not possible. Hence, $\E_1(m,p)$ is strong if, and only if, $m\neq 0$ and the polynomial $mX^2 + 1$ has no roots in $\F_p$.

\item Let $\E_2(m,p) = \mathbb{F}_p\oplus \mathbb{F}_p\oplus \mathbb{F}_p$, with $0 \leq m < p$. For every $k_1,k_2,k_3, t_1,t_2, t_3\in \mathbb{F}_p$, we define
\begin{align*}
&(k_1,k_2,k_3)+(t_1,t_2,t_3):= (k_1+t_1,k_2+t_2,k_3+t_3) \\
&(k_1,k_2,k_3)(t_1,t_2,t_3):= (k_1+t_1, k_2+t_2,k_3+t_3+ mk_1t_1+ k_1t_2 + k_2t_2)
\end{align*}
By Theorem~\ref{teo:(strong)extras-(strong)bilinear}, $\E_2(m,p)$ is an extraspecial left brace, since the map $\phi\colon \F_p^2 \times \F_p^2 \rightarrow \F_p$, given by 
\[ \phi((k_1,k_2), (t_1,t_2)) = mk_1t_1 + k_1t_2 + k_2t_2\] is a bilinear form. Moreover, it holds
\[ (k_1,k_2,k_3) \ast (t_1, t_2,t_3) = (0, 0, mk_1t_1 +k_1t_2+ k_2t_2)\]
Therefore, $\zeta(\E_2(m,p)) = \langle (0,0,1) \rangle_+$. Moreover, $a \ast a \neq 0$ for every $a = (k_1,k_2,k_3) \notin \zeta(\E_2(m,p))$ if, and only if, $m \neq 0$ and $mk_1^2 + k_1k_2 + k_2^2 \neq 0$ for every $k_1,k_2 \in \mathbb{F}_p$ not both zero. Of course, if $k_2 \neq 0$ this is equivalent to say that the polynomial $mX^2 +X + 1$ has no roots in $\mathbb{F}_p$; while if $k_2 = 0$, then $mk_1^2 = 0$ and $m \neq 0$ imply that $k_1 = k_2 = 0$, which is not possible. Hence, $\E_2(m,p)$ is strong if, and only if, $m\neq 0$ and the polynomial $mX^2 + X + 1$ has no roots in $\F_p$.
\end{enumerate}
\end{exemple}

Indeed these are the unique families of finite strong extraspecial left braces.

\begin{teorema}
\label{teo:classif-strong}
Let $E$ be a finite strong extraspecial left brace. Then, $E$ is isomorphic to either $\E_0(m,p)$, $\E_1(m,p)$ or $\E_2(m,p)$ for some $p$ prime and some $0 < m < p$.
\end{teorema}

\begin{proof}
From Theorem~\ref{teo:(strong)extras-(strong)bilinear} and Corollary~\ref{cor:no-isotrop->dim<=2} it follows that $E$ has order at most $p^3$ for some $p$ prime. Moreover, $|E| = p$ is not possible, since $E$ would be an abelian left brace.

Assume that $|E| = p^2$. Then, $\zeta(E) = \langle c \rangle_+$ for some $c \in E$, and $(E,+) = \langle b \rangle_+ \oplus \langle c \rangle_+$, for some $b\in E\setminus \zeta(E)$. Since $E$ is strong, $0 \neq b \ast b \in \zeta(E)$. Write $b \ast b = mc$ with $0 < m < p$. Thus, for every $x = k_1b+ k_2c, y = t_1b +t_2c$, $x \ast y = (mk_1t_1)c$. By Theorem~\ref{teo:(strong)extras-(strong)bilinear}, $E$ is isomorphic to $\E_0(m,p)$.

Assume that $|E| = p^3$. Let $C = \langle c \rangle_+ \leq \zeta(E)$ such that $E/C$ is abelian. Since $E$ is strong, without loss of generality, we can assume that there exists $b\in E \setminus C$ such that $b \ast b = c$. 
By Theorem~\ref{teo:(strong)extras-(strong)bilinear}, $E/C$ can be regarded as a vector space so that
\[ (xC, yC) \mapsto k_{x,y}, \quad \text{where $x \ast y = k_{x,y}c$,}\]
defines a strong non-degenerate bilinear form $\phi\colon E/C\times E/C \rightarrow \F_p$. 

According to Corollary~\ref{cor:2.12phi(x,x)neq0_triangularmatrix}, we can find $a \notin C$ such that $E/C = \langle aC\rangle_+ \oplus \langle bC \rangle_+$, and $0 \neq \phi(aC,aC) = m_1$, $\phi(aC,bC) = m_2$, $\phi(bC,bC) = 1$ and $\phi(bC,aC) = 0$, with $m_1,m_2 \in \F_p$. If $m_2\neq 0$, without loss of generality, we can assume that $m_2 = 1$. 
Hence, we can write $E/C = \langle aC \rangle_+ \oplus \langle bC \rangle_+$ so that $a \ast a = mc$, $a\ast b = \varepsilon c$ and $b \ast b = c$, with $0\neq m\in \F_p$ and $\varepsilon \in \{0,1\}$.

Therefore, for every $x = k_1a + k_2b + k_3c, y=t_1a +t_2b + t_3c \in E\setminus C$, with $k_i, t_i \in \F_p$, $1\leq i \leq 3$,
\[ x \ast y = (mk_1t_1 + \varepsilon k_1t_2 + k_2t_2)c\]
Hence, $E$ is either isomorphic to $\E_1(m,p)$ or $\E_2(m,p)$ by whether $\varepsilon = 1$ or~not.
\end{proof}

\section{Dedekind left braces whose additive group is elementary abelian $p$-group}
\label{sec:dedekind_elem-abel}

In this section we present a structural theorem for Dedekind left braces whose additive group is an elementary abelian $p$-group, $p$ a prime. 

The following theorem provides a sufficient condition in terms of strong extraspecial left braces.

\begin{teorema}
\label{teo:descomp-strong->Dedekind}
Let $A$ be a left brace whose additive group is  an elementary abelian $p$-group for some prime $p$. Suppose that $A = E \oplus Z$ where $E$ is a strong extraspecial brace and $Z \leq \zeta(A)$. Then, $A$ is a Dedekind left brace.
\end{teorema}

\begin{proof}
Let $c\in \zeta(E)$ such that the quotient  $E/\langle c \rangle_+$ is abelian. By previous Remark~\ref{nota:extrasp-socle=centre}, $\zeta(A) = \zeta(E)\oplus Z = \Soc(A)$, and $\Soc_2(A) = A = \Ann_A(\Soc(A))$. Let $a,a' \in A$ and $e,e'\in E$, $z,z'\in Z$ such that $a = e+z$ and $a' = e'+z'$. It holds that
\[ a\ast a' = (e+z)\ast e' + (e+z)\ast z' = (e+z) \ast e'= (ez) \ast e' = e\ast e' \in \langle c \rangle_+\]
Thus, $A \ast A = E \ast E = \langle c \rangle_+$ and $A/\langle c \rangle_+$ is abelian. 

Let $a \in A$. If $a \in \zeta(A)$, then $\langle a \rangle = \langle a \rangle_+$ is an ideal of $A$. If $a \notin \zeta(A)$, then $\langle c \rangle_+ = \langle a \ast a \rangle_+ \leq \langle a \rangle$. Since $A/ \langle c \rangle_+$ is abelian, $\langle a \rangle/\langle c \rangle_+$ is an ideal of $A/\langle c \rangle_+$ and so, $\langle a \rangle$ is an ideal of $A$. Hence, every subbrace of $A$ is ideal. 
\end{proof}

\begin{corolari}
\label{cor:E-ded<->arrels-pol}
$\E_0(m,p)$ is a Dedekind left brace for every prime $p$ and $0 < m < p$. Moreover, $\E_1(m,p)$ and $\E_2(m,p)$ are Dedekind if, and only if, $m \neq 0$ and the polynomials $mX^2 +1$ and $mX^2 + X + 1$ have respectively no roots in~$\mathbb{F}_p$. 
\end{corolari}

\begin{proof}
Example~\ref{ex:extrasp} and Theorem~\ref{teo:descomp-strong->Dedekind} yield the result for the family $\E_0(m,p)$ and also yield the sufficient condition for the families $\E_1(m,p)$ and $\E_2(m,p)$. For the necessary condition we can argue on the family $\E_2(m,p)$, as the argument for $\E_1(m,p)$ is equivalent.

Assume that $\E_2(m,p)$ is Dedekind and that there exists $0\neq k \in \mathbb{F}_p$ such that $mk^2 + k + 1 = 0$. Consider $x = (k,1,0) \in \E_2(m,p) \setminus \zeta(\E_2(m,p))$. According to Example~\ref{ex:extrasp}, $x \ast x = (0,0,0)$. By Lemma~\ref{lema:a·a=0ciclic}, $\langle x \rangle = \langle x \rangle_+$. 

Take $y = (1,0,0)$. Then, $x \ast y = (0,0, mk)$ and $mk \neq 0$, as $m,k \neq 0$. Thus, $x \ast y\notin \langle x \rangle$ and so, $\langle x \rangle$ is not an ideal; a contradiction.
\end{proof}

\begin{exemple}
For specific primes $p$, Corollary~\ref{cor:E-ded<->arrels-pol} allows to establish conditions on the parameter $m$ to characterise Dedekind left braces. For example, for $p = 5$, $\E_1(m,p)$ is Dedekind if, and only if, $0 < m < p$ and $m\notin \{1,4\}$.
\end{exemple}

The rest of the section is devoted to show the structural part of the converse of this theorem. To this aim, it is imperative to study the first two terms of the socle series of a Dedekind left brace. Moreover, Theorems~\ref{teo:(strong)extras-(strong)bilinear} and~\ref{teo:classif-strong}, and Corollary~\ref{cor:E-ded<->arrels-pol} play also a prominent role.

\emph{From now on, for an arbitrary left brace $A$, $S$ denotes $\Soc(A)$ and $C$ denotes $\Soc_2(A)$}.

\begin{lema}
\label{lema:dedekind+Celemabelian-prop}
Let $A$ be a Dedekind left brace. Suppose that $(C,+)$ is an elementary abelian $p$-group for some prime $p$. Let $d \in C \setminus S$. Then, the following statements hold
\begin{enumerate}
\item $\langle d\ast d \rangle = \langle d \ast d \rangle_+$ and $\langle d \rangle = \langle d \rangle_+ \oplus \langle d\ast d\rangle_+$;
\item $d \ast a \in \langle d \ast d \rangle_+$ for every $a\in A$;
\item $\langle d \ast d\rangle_+ = \langle c \ast c \rangle_+$ for every $c\in C \setminus S$.
\end{enumerate}
\end{lema}

\begin{proof}
1. Since $d\in C$, we have that $u:= d\ast d \in S$. If $u = 0$, then $\langle d \rangle = \langle d \rangle_+$ by Lemma~\ref{lema:a·a=0ciclic}. If $u \neq 0$, then $u\ast u = 0$ and by Lemma~\ref{lema:a·a=0ciclic}, it holds $U:= \langle u \rangle = \langle u \rangle_+$ is an ideal. Using again Lemma~\ref{lema:a·a=0ciclic}, it follows $\langle d+U\rangle = \langle d+ U\rangle_+ = \langle d \rangle_+ + U$. Therefore, $\langle d \rangle  \leq \langle d \rangle_+ + U$. The other inclusion is trivial. Since $C$ is elementary abelian and $d\in C\setminus S$, $\langle d \rangle_+ \cap \langle u \rangle_+ = \{0\}$ and the equality holds. 

2. We have that $\langle d \rangle$ is an ideal. Thus, $d \ast a \in \langle d \rangle \cap S$, for every $a\in A$. Assume that $x = kd + tu \in \langle d \rangle \cap S$, for some $0 \leq k,t < p$. Thus, $kd \in \langle d\rangle_+ \cap S$. Therefore, $k = 0$ as $d\notin S$. Hence, $\langle d \rangle \cap S = \langle u \rangle_+$ as desired.

3. Let $c \in C\setminus S$. Then, $v:= c \ast c \neq 0$, otherwise $c \in S$ by item~$2$. Thus, $\langle c \rangle = \langle c  \rangle_+ + \langle v\rangle_+$. Suppose that $\langle u \rangle_+ \neq \langle v \rangle_+$. It follows that $\langle u\rangle_+ \cap \langle v\rangle_+ = 0$, as $(C,+)$ is elementary abelian. Therefore, $d \ast c = c \ast d = 0$ by item~$2$.

Set $e:= c+d$. Then, $d \ast e = u$ and $c\ast e = v$. Thus, $u,v \in \langle e \rangle$, as $\langle e \rangle$ is an ideal. By item $1$, either $\langle e \rangle = \langle e \rangle_+$ is a cyclic $p$-group or $\langle e \rangle = \langle e \rangle_+ \oplus \langle e \ast e \rangle_+$. In the former case, it holds $\langle u \rangle_+ = \langle v \rangle_+$; in the latter case, $u,v \in \langle e \rangle \cap S = \langle e \ast e \rangle_+$, so that $\langle u \rangle_+ = \langle v \rangle_+$ also holds.
\end{proof}




\begin{corolari}
\label{cor:2.5_C=L+K_C/B<soc(A/B)}
Let $A$ be a Dedekind left brace. Suppose that $(C,+)$ is an elementary abelian $p$-group for some prime $p$. Then, $C = E \oplus Z$ where $E$, $Z$ are ideals of $A$ satisfying
\begin{enumerate}
\item $Z \leq S$;
\item $B:=C \ast C = E \ast E \leq \Soc(A)$ is an ideal of $A$ of order $p$ so that $C/B = \Soc(A/B)$;
\item $u\ast u \neq 0$ for all $u \in E \setminus B$.
\end{enumerate}
\end{corolari}

\begin{proof}
Firstly, observe that $C \ast C = \langle d \ast u\mid d, u\in C \rangle_+ \subseteq S$. Since $S$ is abelian, $C \ast C \leq S$ and, therefore, $C \ast C$ is an ideal of $A$ as $A$ is Dedekind. 

Let $d \in C$. If $d\in S$, then $d\ast a = 0$ for every $a\in A$. Suppose that $d\notin S$. By Lemma~\ref{lema:dedekind+Celemabelian-prop}, $d\ast u \in \langle d \ast d \rangle_+$, for every $u \in C$, and $\langle d \ast d \rangle_+ = \langle c \ast c \rangle_+$, for every $c\in C\setminus S$. Hence, $C \ast C = \langle d \ast d \rangle_+$ has order $p$.

Now, since $S \leq C$, $(S,+)$ is elementary abelian and $S = B \oplus Z$ for some subbrace $Z$ of $S$. Thus, $Z$ is also an ideal of $A$. Since the quotient $C/B$ is abelian, $S/B$ has a complement in $C/B$, say $E/B$, where $E$ is also an ideal of $A$. Thus, we can write $C/B = S/B \oplus E/B$ and it follows $C = E \oplus Z$. 

Let $u \in E \setminus B$. If $u \ast u = 0$ then $u\in S$ (otherwise, Lemma~\ref{lema:dedekind+Celemabelian-prop} leads to a contradiction). Thus, $u \in S \cap E = B$, a contradiction. Therefore, $u \ast u \neq 0$ and again by Lemma~\ref{lema:dedekind+Celemabelian-prop}, $E \ast E = C \ast C$.

Finally, let $a \in A$ and $d \in C\setminus B$.  If $d\notin S$, by Lemma~\ref{lema:dedekind+Celemabelian-prop}, $d \ast a \in \langle d \ast d \rangle_+ = C \ast C$. Hence, $C/B \leq \Soc(A/B)$. On the other hand, let $u+B \in \Soc(A/B)$. Then, for every $a \in A$, $u\ast a \in B$. Since $B \leq S$, it follows that $u + S \in \Soc(A/S) = C/S$. Hence, $C/B = \Soc(A/B)$. 
\end{proof}

The following theorem shows that $\Soc_2(A) = A$ holds for 
mutipermutational braces with elementary abelian additive $p$-group. 

\begin{teorema}
\label{teo:4.7_Ahypersoc+adelementp+sub=id->A=Soc2}
Let $A$ be a 
multipermutational Dedekind left brace whose additive group is an elementary abelian $p$-group for some prime $p$. Then, $A = \Soc_2(A)$.
\end{teorema}

\begin{proof}
Suppose that $A\neq \Soc_2(A)$. Then, $\Soc_3(A) \neq \Soc_2(A)$. Call $S = \Soc_1(A)$, $C = \Soc_2(A)$, and $\bar{A} = \Soc_3(A)$. By Corollary~\ref{cor:2.5_C=L+K_C/B<soc(A/B)}, $B = C \ast C \leq S$ has order $p$ and $C/B = \Soc(A/B)$. Hence, $\Soc_2(A/B) = \bar{A}/B$. Moreover,  Proposition~\ref{prop:dedekind->Soc_a=Z_a} yields $S = \zeta(A)$ and $\Soc(A/B) = \zeta(A/B)$. Thus, $B \leq \zeta(A)$.

Applying Corollary~\ref{cor:2.5_C=L+K_C/B<soc(A/B)} to $A/B$, we can write $\bar{A}/B=  (E/B)\oplus (Z/B)$, where $E$ and $Z$ are ideals of $A$ such that 
\begin{itemize}
\item $Z/B \leq \zeta(A/B)$;   
\item $D/B:= (\bar{A} \ast \bar{A})/B = (E \ast E)/B \leq \zeta(A/B)$ is an ideal of $A/B$ of order $p$ and 
\[ \text{$(\bar{A}/B)/(D/B)$ is abelian, or equivalently, $\bar{A}/D$ is abelian;}\]
\item $u\ast u \notin B$ for every $u\in E\setminus D$. 
\end{itemize}

Let $a \in \bar{A} \setminus D$ and set  $d:= a \ast a \neq 0$. Thus, $d \in D \leq C$ and $\langle d + B \rangle_+ = D/B$. If $d\ast d = 0$ then $d \in S$ (otherwise, Lemma~\ref{lema:dedekind+Celemabelian-prop} yields a contradiction). Thus, $D = \langle d \rangle_+ \oplus B \leq S$. Therefore, $\bar{A}\ast \bar{A} \leq S$ and so $\Soc_2(A) = \bar{A}$, a contradiction. Then, we can assume that there exists $a \in \bar{A} \setminus D$ such that $a \ast a = d \in D\leq C\setminus S$ and $d \ast d \neq 0$. Set $0 \neq b:= d \ast d \in C \ast C = B$. Thus, $\langle b \rangle_+ = B \leq \zeta(A)$, as $B$ has prime order. Hence, it follows that $\langle a \rangle = \langle a \rangle_+ \oplus \langle d \rangle_+ \oplus \langle b \rangle_+$ is a subbrace, and therefore, an ideal of $A$.

Take $0 \leq \alpha, \beta < p$ such that $a \ast d = \alpha b = b_1$ and $d \ast a = \beta b = b_2$, and let $u:= a + d + b$. We know that $uu = u \ast u + 2u$. Since $d+B \in \zeta(A/B)$, it follows that $a+d + B = ad +B = adB$. Thus, there exists $w\in B \leq \zeta(A)$ such that $a+d = adw$. Then, it holds that
\begin{align*}
u \ast u & = (a+d+b)\ast (a+d+b) = (adw) \ast (a+d+b) = \\
& = (ad) \ast (a+d+b) = (ad) \ast a + (ad) \ast d + (ad) \ast b = \\
& = a \ast (d \ast a) + d \ast a + a \ast a + a \ast (d \ast d) + d \ast d + a \ast d = \\
& = a \ast b_2 + b_2 + d + a \ast b + b + b_1 = \\
& = b_2 + d + b + b_1 = d + (\alpha + \beta + 1) b 
\end{align*}
Thus, $uu = u \ast u + 2u = 2a + 3d +  (3 + \alpha + \beta)b$. 

Observe that $a^2 = 2a + a \ast a = 2a + d$, and similarly, $d^2 = 2d + b$ so that 
\[ d^3 = d(2d+b) = 2d^2 - 2d + db = 4d + 2b - 2d + d + d \ast b + b = 3d + 3b\]
Since $d+B \in \zeta(A/B)$ and $B \leq \zeta(A)$, there exists $w_1 \in B$ such that
\begin{align*}
2a & = a^2 - d = a^2 + (p-1)d = a^2d^{p-1}w_1\\
3d & = d^3 - 3b = d^3 + (p-3)b = d^3b^{p-3}
\end{align*}
Moreover, we claim that $d^n \ast a = nb_2$ and $d^n \ast d = nb$, for every positive integer $n$. Indeed, if we assume that $d^n \ast a = nb_2$ and $d^n\ast d = nb$, for some $n \geq 1$, we get 
\begin{align*}
d^{n+1} \ast a & = d \ast (d^n \ast a) + d^n \ast a + d \ast a = 0 + d^n\ast a + d \ast a = (n+1)b_2\\
d^{n+1} \ast d & = d \ast (d^n \ast d) + d^n \ast d + d \ast d = 0 + nb  + b = (n+1) b
\end{align*}
Since $w_1, b_2,b \in B\leq \zeta(A)$, we obtain
\begin{align}
(2a) \ast a & = (a^2d^{p-1}w_1)\ast a = (a^2d^{p-1}) \ast a = \nonumber \\
& = a^2 \ast (d^{p-1} \ast a) + d^{p-1}\ast a + a^2 \ast a =  \nonumber \\ 
& = a^2 \ast ((p-1)b_2) + (p-1)b_2 + a \ast (a \ast a) + 2a\ast a = \nonumber \\
& = -b_2 + b_1 + 2d = 2d + (\alpha - \beta)b  \label{eq:2a ast a} \\
(2a) \ast d & = (a^2d^{p-1}w_1)\ast d = (a^2d^{p-1}) \ast d = \nonumber \\ 
& = a^2 \ast (d^{p-1}\ast d) + d^{p-1} \ast d  + a^2 \ast d =  \nonumber  \\
& = a^2 \ast ((p-1)b) + (p-1)b +  a \ast (a \ast d) + 2a \ast d = \nonumber\\
& = -b + 2b_1 = (2\alpha - 1)b  \label{eq:2a ast d} \\
(3d) \ast a & = (d^3b^{p-3}) \ast a = d^3 \ast a = 3b_2 = 3\beta b \label{eq:3d ast a}\\
(3d) \ast d & = (d^3b^{p-3}) \ast d = d^3 \ast d = 3b \label{eq:3d ast d}
\end{align}

Now, since $d+B \in \zeta(A/B)$, we can write $uu = 2a + 3d + (3+\alpha + \beta)b = (2a)(3d)w_2$, for some $w_2 \in B$. Then, applying Equations~\eqref{eq:2a ast a}--\eqref{eq:3d ast d} and the fact that $w_2,b,b_1,b_2\in \zeta(A)$, we obtain
\begin{align*}
(uu) \ast u & = \big((2a)(3d)w_2\big)\ast u = \big((2a)(3d)\big) \ast u = \\
& = (2a) \ast \big((3d) \ast u\big) + (3d) \ast u + (2a) \ast u \\
& = (2a) \ast \big( (3d) \ast a + (3d) \ast d + (3d) \ast b\big) + \\
& +  (3d) \ast a + (3d) \ast d + (3d) \ast b +  (2a) \ast a + (2a) \ast d + (2a) \ast b = \\
& = (2a) \ast (3\beta b + 3b)  + 3\beta b + 3b + 2d + (\alpha-\beta)b + (2\alpha-1)b = \\
& = 2d + (2 + 3\alpha + 2\beta)b
\end{align*}
Thus, $(uu)u = (uu) \ast u + uu + u = 3a + 6d + (6 + 4\alpha + 3\beta)b$. On the other hand, similarly we see
\begin{align*}
u \ast (uu) & = (adw) \ast (uu) = (ad) \ast (uu) = \\
& = a \ast \big(d \ast (uu)\big) + d \ast (uu)  + a \ast (uu) =\\
& = a \ast \Big(d \ast (2a) + d\ast (3d) +  d\ast \big((3 + \alpha + \beta)b\big)\Big) + \\
& + d \ast (2a) + d\ast (3d) +  d\ast \big((3 + \alpha + \beta)b\big) + \\
& + a \ast (2a) + a\ast (3d) +  a\ast \big((3 + \alpha + \beta)b\big) = \\
& = 2b_2 + 3b + 2d + 3b_1 = 2d + (3+ 3\alpha + 2\beta)b
\end{align*}
and so, $u(uu) = u\ast (uu) + uu + u = 3a + 6d + (7 + 4\alpha + 3 \beta)b$. Therefore, $(uu)u - u(uu) = b \neq 0$, and we arrive to a contradiction. Hence, the theorem holds.
\end{proof}

\begin{corolari}
Let $A$ be a finite Dedekind left brace whose additive group is an elementary abelian $p$-group for some prime $p$. Then, $A$ is centrally nilpotent of level $2$.
\end{corolari}

\begin{proof}
It follows from Theorems~\ref{teo:Dedekind->central-nilp} and~\ref{teo:4.7_Ahypersoc+adelementp+sub=id->A=Soc2}.
\end{proof}

Now, we can state the converse of Theorem~\ref{teo:descomp-strong->Dedekind}.
\begin{teorema}
\label{teo:Ahypersoc+Dedekind->descomp-strong}
Let $A$ be a 
multipermutational Dedekind left brace such that $(A,+)$ is an elementary abelian $p$-group, $p$ a prime. Then, $A$ is an extraspecial brace satisfying
\begin{enumerate}
\item $A = E \oplus Z$ where $E$ and $Z$ are ideals of $A$, with $E$ strong extraspecial;
\item $A \ast A = E \ast E$ has order $p$;
\item $Z \oplus (E \ast E) = \zeta(A) = \Soc(A)$.
\end{enumerate}
\end{teorema}

\begin{proof}
It follows from Theorem~\ref{teo:4.7_Ahypersoc+adelementp+sub=id->A=Soc2} and Corollary~\ref{cor:2.5_C=L+K_C/B<soc(A/B)}.
\end{proof}

In the finite case, applying Theorem~\ref{teo:classif-strong} and Corollary~\ref{cor:E-ded<->arrels-pol}, we can complete the description of Theorem~\ref{teo:Ahypersoc+Dedekind->descomp-strong}.

\begin{corolari}
\label{cor:Ahypersoc+finit++Dedekind->E0-1-2+polinomis}
Let $A$ be a finite Dedekind left brace such that $(A,+)$ is an elementary abelian $p$-group for some $p$ prime. Then, $A$ is an extraspecial brace $A$ satisfying
\begin{enumerate}
\item $A = E \oplus Z$ where $E$ and $Z$ are ideals of $A$, and $E$ is a strong extraspecial left brace of order at most~$p^3$;
\item $A \ast A = E \ast E$ has order $p$;
\item $Z \oplus (E \ast E) = \zeta(A) = \Soc(A)$.
\end{enumerate}
Furthermore, the following hold
\begin{enumerate}
\item if $|E| = p^2$, then $E$ is isomorphic to $\E_0(m,p)$ with $0 < m < p$;
\item if $|E| = p^3$, then either $E$ is isomorphic to $\E_1(m,p)$ with $0 < m < p$ such that $mX^2 +1$ has no roots in $\F_p$, or $E$ is isomorphic to $\E_2(m,p)$ with $0 < m < p$ such that $mX^2 + X +1$ has no roots in $\F_p$.
\end{enumerate}
\end{corolari}

\section*{Acknowledgements}
The third author is very grateful to the Conselleria d'Innovaci\'o, Universitats, Ci\`encia i
Societat Digital of the Generalitat (Valencian Community, Spain) and the Universitat de
Val\`encia for their financial support and grant to host researchers affected by the war in
Ukraine in research centres of the Valencian Community. He is sincerely grateful to the
first and second authors for their hospitality, support and care. The third author would like also to thank the Isaac Newton Institute for Mathematical Sciences, Cambridge, for support and hospitality during the Solidarity Supplementary Grant Program. This work was supported by EPSRC grant no EP/R014604/1. He is sincerely grateful to Agata Smoktunowicz.

We would like to thank the referee for the careful reading and the very useful comments that help us to improve this paper.

\section*{Competing Interests}
The authors have no relevant financial or non-financial interests to disclose.

\section*{Author Contributions}
All authors contributed to the study conception and design. 

\section*{Data Availibility Statement}
Data sharing not applicable to this article as no datasets were generated or analysed during the current study.

\bibliographystyle{plain}
\bibliography{bibgroup}
\end{document}